\documentclass{amsart}
\usepackage{amssymb,latexsym,amsmath,amscd,amsthm,amsfonts}
\begin{document}



\title {Reflexivity of Operator Algebras of finite split strict multiplicity}


\author{Raluca Dumitru, Costel Peligrad, Bogdan Visinescu}

\address{Raluca Dumitru, University of North Florida}
\address{Istitute of Mathematics of the Romanian Academy}
\email{raluca.dumitru@unf.edu}
\address{Costel Peligrad, University of Cincinnati}
\email{costel.peligrad@uc.edu}
\address{Bogdan Visinescu, University of North Florida}
\address{Istitute of Mathematics of the Romanian Academy}
\email{b.visinescu@unf.edu}
\thanks{}
\maketitle

\begin{abstract}
We study the invariant subspaces of abelian operator algebras of finite split strict multiplicity. We give sufficient conditions for the reflexivity and hereditary reflexivity of these algebras.

\end{abstract}

\

\

\newtheorem{definition}{Definition}
\newtheorem{lemma}[definition]{Lemma}

\theoremstyle{definition}
\newtheorem{remark}[definition]{Remark}
\theoremstyle{plain}
\newtheorem{proposition}[definition]{Proposition}
\newtheorem{theorem}[definition]{Theorem}
\newtheorem{corollary}[definition]{Corollary}

Let $A\subset B(X)$ be a strongly closed algebra of operators on a Banach space $X$. We denote by $Lat(A)$ the set of all closed subspaces of $X$ that are invariant for every $a\in A$. Let $algLat(A)=\{ T\in B(X)\;|\;TK\subseteq K,$ for every $K\in Lat(A)\}$.

An algebra $A$ is called reflexive if $A=algLat(A)$. More generally, a subspace $\mathcal{L}\subseteq B(X)$ is called reflexive if for every $T\in B(X)$ such that $Tx \in \overline{\mathcal{L}x}$, for every $x\in X$, it follows that $T\in \mathcal{L}$. An algebra $A$ is called hereditarily reflexive if every strongly closed subspace $\mathcal{L}\subseteq A$ is reflexive.

In \cite{Lambert}, A. Lambert initiated the study of strictly cyclic algebras and proved that a unital, strictly cyclic, semisimple operator algebra on a Hilbert space is reflexive. In \cite{Peligrad} the result of Lambert is extended to the case of Banach spaces and is proven that every norm closed, abelian, unital, semisimple, strictly cyclic algebra is in addition hereditarily reflexive. In this paper we investigate the similar problem for algebras of finite strict multiplicity. In \cite{Herrero} an algebra $A\subset B(X)$ is said to be of strict multiplicity $p\in\mathbb{N}$ (denoted by $s.m.(A)=p$), if there are $x_{1},x_{2},...,x_{p}\in X$ such that $X= $ linear span $\{ax_{i}\;|\; a\in A;\; i=1,2,...,p\}$.

We will consider a special case of algebras of finite strict multiplicity, namely algebras of finite split strict multiplicity. An algebra $A\subset B(X)$ is said to be of split strict multiplicity $p\in \mathbb{N}$ (denoted by $s.s.m(A)=p$) if there are $x_{1},x_{2},...,x_{p}\in X$ such that the subspaces $\{ax_{i}\;|\;a\in A\}=X_{i}$ are closed and $X=\underset{i=1}{\overset{p}\sum}\oplus X_{i}$.

Proposition \ref{ssmprop} and Remark \ref{rem2} below provide a characterization of abelian operator algebras of split strict multiplicity 2. In particular, if $A\subseteq B(X)$ is strictly cyclic, then $A^{(2)}=\{a\oplus a\;|\;a\in A\}\subseteq B(X\oplus X)$ is the simplest example of an algebra of split strict multiplicity 2.

\begin{proposition}\label{ssmprop} Let $A\subseteq B(X)$ be an abelian, unital, norm closed algebra of strict split multiplicity $p$. Then there exist closed ideals $I_{k}\subseteq A$, $k=1,2,...,p$ and Banach space isomorphisms $S_{k}:A/I_{k}\to Ax_{k}$ such that:

\begin{itemize}
\item[(i)] $S_{k}(\varphi_{k}(a))=ax_{k}.$
\item[(ii)] $ S=\underset{k=1}{\overset{p}\sum}\oplus S_{k}:\underset{k=1}{\overset{p}\sum}\oplus A/I_{k} \to X$ is a Banach space isomorphism.
\item[(iii)] The mapping $T:A\to B(\underset{k=1}{\overset{p}\sum}\oplus A/I_{k})$ defined by $Ta=S^{-1}aS$ is a Banach algebra isomorphism.
\item[(iv)] If $p=2$, $I_{1}+I_{2}$ is closed.
\end{itemize}

\end{proposition}

\begin{proof}
Let $I_{k}=\{a\in A\;|\;ax_{k}=0\}$, $k=1,2,...,p.$ Clearly, $I_{k}$ are closed ideals of $A$. Let $S_{k}:A/I_{k}\to Ax_{k}$ be defined by $S_{k}(\varphi_{k}(a))=ax_{k}$. Then $S_{k}$ is well defined, one to one and onto for all $k=1,2,...,p.$ To see that $S_{k}$ is continuous, let $\{a_{n}\}_{n}\subset A$ be such that $\varphi_{k}(a_{n})\underset{n}{\to}0$. Then there is a sequence $\{i_{k,n}\}_{n}\subset I_{k}$ such that $a_{n}+i_{k,n}\underset{n}{\to}0$ in the norm of $A$. It the follows that $(a_{n}+i_{k,n})x_{k}\underset{n}{\to}0$. Since $i_{k,n}\in I_{k}$ it follows that $a_{n}x_{k}\underset{n}{\to}0$, so $S_{k}$ is a continuous, one-to-one, onto mapping between the Banach spaces $A/I_{k}$ and $Ax_{k}$. By the open mapping theorem, $S_{k}$ is a Banach space isomorphism.

Conditions (ii) and (iii) are straightforward to ckeck. We verify next condition (iv).

Let $a_{n}=i_{1,n}+i_{2,n}\underset{n}{\to}a_{0}$, where $\{i_{1,n}\}\subset I_{1}$ and $\{i_{2,n}\}\subset I_{2}$. Then, in particular, the restriction ${a_{n}|}_{X_{1}}={i_{2,n}|}_{X_{1}}\underset{n}{\to}{a_{0}|}_{X_{1}}$. On the other hand, ${i_{2,n}|}_{X_{2}}=0$ for every $n\in \mathbb{N}$. It then follows that the sequence $\{i_{2,n}\}$ is convergent in norm. Since $I_{2}$ is norm closed, it follows that there exists $i_{2}\in I_{2}$ such that $i_{2,n}\underset{n}{\to}i_{2}$. Hence $a_{n}-i_{2,n}\to a_{0}-i_{2}$ in norm and since $a_{n}-i_{2,n}=i_{1,n}\in I_{1}$ and $I_{1}$ is closed, we have $a_{0}-i_{2}=i_{1}\in I_{1}$, and thus $a_{0}=i_{1}+i_{2}$ and $I_{1}+I_{2}$ is closed as claimed.

\end{proof}

\begin{remark}\label{rem2} Let $A$ be an abelian, unital Banach algebra and $I_{1}, I_{2}$ be two closed ideals such that $I_{1}+I_{2}$ is closed and $I_{1}\cap I_{2}=(0)$. Then the mapping $a\to T_{a}\in B(A/I_{1}\oplus A/I_{2})$ where $T_{a}(\varphi_{1}(a_{1})\oplus \varphi_{2}(a_{2}))=\varphi_{1}(aa_{1})\oplus \varphi_{2}(aa_{2})$ is a continuous embedding of $A$ into $B(A/I_{1}\oplus A/I_{2})$.

By Proposition \ref{ssmprop}, any abelian, unital, norm closed $A\subseteq B(X)$, ($X$ Banach space) with $s.s.m(A)=2$ is spatially isomorphic with $\{T_{a}\;|\; a\in A\}$ if  $I_{k}=\{a\in A\; |\; ax_{k}=0\}$, for $k=1,2$.

\end{remark}

\begin{remark}\label{rem3}

(i) Let $A\subseteq B(X)$ be an abelian, unital, norm closed subalgebra with $s.s.m.(A)=2$. Then, the subspace $D_{2}=\{ax_{1}\oplus ax_{2}\;|\;a\in A\}\subseteq X$ is closed. 

Indeed, let $\{a_{n}\}\subset A$ be such that $a_{n}x_{1}\oplus a_{n}x_{2}\to ax_{1}\oplus bx_{2}\in X$. Then, by Proposition \ref{ssmprop}, $\varphi_{1}(a_{n}\to \varphi_{1}(a)$ and $ \varphi_{2}(a_{n})\to \varphi_{2}(b)$. Therefore, there are sequences $\{i_{1,n}\}\subset I_{1}$, $\{i_{2,n}\}\subset I_{2}$ such that $a_{n}+i_{1,n}\underset{n}{\to} a$ and $a_{n}+i_{2,n}\underset{n}{\to}b$. Hence $i_{1,n}-i_{2,n}\underset{n}{\to}a-b$. By Proposition \ref{ssmprop} (iv), $a-b\in I_{1}+I_{2}$, so there exist $i_{1}\in I_{1}$, $i_{2}\in I_{2}$ such that $a-b=i_{1}-i_{2}$. Hence $a-i_{1}=b-i_{2}$. Let $a_{0}= a-i_{1}=b-i_{2}$. Then $a_{0}x_{1}=ax_{1}$ and $a_{0}x_{2}=bx_{2}$, so $a_{n}x_{1}\oplus a_{n}x_{2} \underset{n}{\to} a_{0}x_{1}\oplus a_{0}x_{2}\in D_{2}$.  

(ii) Let $A\subseteq B(X)$ be an abelian, unital, strongly closed subalgebra with $s.s.m.(A)=p\in\mathbb{N}$. Then $D_{p}=\{ax_{1}\oplus ax_{2} \oplus ...\oplus ax_{p}\;|\;a\in A\}\subseteq X$ is closed.

Indeed, let $\{a_{n}\}\subset A$ be such that $a_{n}x_{1}\oplus a_{n}x_{2} \oplus ...\oplus a_{n}x_{p}\underset{n}{\to}b_{1}x_{1}\oplus ...\oplus b_{p}x_{p}\in X$. Then, since $A$ is abelian, $a_{n}c_{j}x_{j}\to b_{j}c_{j}x_{j}$ for every $c_{j}\in A$, $j=1,2,...,p$. Hence $a_{n}(\underset{j=1}{\overset{p}{\sum}}c_{j}x_{j})\to \underset{j=1}{\overset{p}{\sum}}b_{j}c_{j}x_{j}$. Therefore the sequence $\{a_{n}x\}$ is convergent for every $x\in X$. Since $A$ is strongly closed, there is $a_{0}\in A$ such that $a_{n}x\to a_{0}x$ for all $x\in X$. Hence $a_{n}x_{1}\oplus a_{n}x_{2} \oplus ...\oplus a_{n}x_{p}=a_{n}(x_{1}\oplus ...\oplus x_{p})\to a_{0}(x_{1}\oplus ...\oplus x_{p})=a_{0}x_{1}\oplus a_{0}x_{2} \oplus ...\oplus a_{0}x_{p}\in D_{p}$

(iii) Let $A\subseteq B(X)$ be an abelian, unital, norm closed subalgebra with $s.s.m.(A)=p\in\mathbb{N}$. Let $\mathcal{M}\subseteq A$ be a maximal ideal. Then the subspaces $\mathcal{M}x_{j}=\{mx_{j}\;|\;m\in \mathcal{M}\}\subseteq X$, $j=1,...,p$, are closed. 

Indeed, the statement follows if we show that $\mathcal{M}+I_{j}$ is closed. If $I_{j}\subseteq \mathcal{M}$ then $\mathcal{M}+I_{j}=\mathcal{M}$. If $I_{j}\not\subseteq \mathcal{M}$ then $\mathcal{M}+I_{j}=A$. In either case $\mathcal{M}+I_{j}$ is closed.

\end{remark}

In what follows we will denote by $\mathcal{R}_{j}$ the radical of the (abelian, unital) Banach algebra $A/I_{j}$, i.e. the intersection of all the maximal ideals of $A/I_{j}$. The next result shows that every abelian, unital, strongly closed subalgebra $A\subseteq B(X)$, with $s.s.m.(A)=p\in\mathbb{N}$ is reflexive modulo the radicals. 

By Proposition \ref{ssmprop} we can assume that $X=\underset{j=1}{\overset{p}\sum}\oplus A/I_{j}$ and $T_{a}(\underset{j=1}{\overset{p}\sum}\oplus \varphi _{j}(a_{j}))=\underset{j=1}{\overset{p}\sum}\oplus \varphi _{j}(aa_{j})$ for $a\in A$.

\begin{theorem}\label{thm4} Let $A\subseteq B(X)$ be an abelian, unital, strongly closed subalgebra with $s.s.m.(A)=p\in\mathbb{N}$. For every $T\in algLat(A)$ there is $a_{0}\in A$ such that $(T-T_{a_{0}})(x)\in \underset{j=1}{\overset{p}\sum}\oplus \mathcal{R}_{j}$ for every $x\in X$.
\end{theorem}

\begin{proof}

With the identification $X=\underset{j=1}{\overset{p}\sum}\oplus A/I_{j}$ and $A=\{T_{a}\;|\; a\in A\}$ from Proposition \ref{ssmprop}, let $T\in algLat(A)$. By Remark \ref{rem3} (ii) , $D_{p}\in Lat(A)$. Since $\underset{j=1}{\overset{p}\sum}\oplus \varphi _{j}(1)\in D_{p}$, there is $a_{0} \in A$ such that $T(\underset{j=1}{\overset{p}\sum}\oplus \varphi _{j}(1))=\underset{j=1}{\overset{p}\sum}\oplus \varphi _{j}(a_{0})$. In particular, if $j\in\{1,2,...,p\}$, $T(\varphi_{j}(1))=\varphi_{j}(a_{0})$. Let $\mathcal{M}\subseteq A$ be a maximal ideal and $\Psi$ the corresponding multiplicative functional on $A$ with $ker\Psi =\mathcal{M}$.

Let $a\in A$. Then $x=\Psi (a)1-a\in \mathcal{M}$ since $\Psi(x)=0$. By Remark \ref{rem3} (iii), $\varphi_{j}(\mathcal{M})\in Lat(A)$. Therefore $\varphi_{j}(x)=\Psi(a)\varphi_{j}(1)-\varphi_{j}(a)\in\varphi_{j}(\mathcal{M})$ and $T(\varphi_{j}(x))=\Psi(a)T(\varphi_{j}(1))-T(\varphi_{j}(a))\in \varphi_{j}(\mathcal{M})$. Hence
\begin{equation}\label{1}
\Psi(a)\varphi_{j}(a_{0})-T(\varphi_{j}(a))\in \varphi_{j}(\mathcal{M})
\end{equation}

Since $T\in algLat(A)$, $T(\varphi_{j}(A))\subseteq \varphi_{j}(A)$. Let $b\in A$ be such that $T(\varphi_{j}(a))=\varphi_{j}(b)$. Using relation \ref{1} above we get
\begin{equation}\label{2}
\Psi(a)a_{0}-b\in \mathcal{M}+I_{j} \text{ or}
\end{equation}
\begin{equation}\label{3}
\Psi(a)a_{0}-b+i\in \mathcal{M}\text{ for some }i\in I_{j}.
\end{equation}
Applying $\Psi$ to relation \ref{3} we get
\begin{equation}\label{4}
\Psi(aa_{0}-b+i)=0.
\end{equation}
Hence $aa_{0}-b\in \mathcal{M}+I_{j}$. It follows that $\varphi_{j}(aa_{0})-T(\varphi_{j}(a))\in\varphi_{j}(\mathcal{M})$. Since this holds in particular for every maximal ideal $\mathcal{M}\supseteq I_{j}$, it follows that $T(\varphi_{j}(a))-T_{a_{0}}(\varphi_{j}(a))\in \mathcal{R}_{j}$ and the proof is complete.

\end{proof}

Notice that in the above theorem we assumed that the algebra $A\subseteq B(X)$ is strongly closed. The next result shows that if $s.s.m.(A)=2$ we can assume that $A$ ia only norm closed.

\begin{corollary} Let $A\subseteq B(X)$ be an abelian, unital, norm closed subalgebra with $s.s.m.(A)=2$. For every $T\in algLat(A)$ there is $a_{0}\in A$ such that $(T-T_{a_{0}})(x)\in \mathcal{R}_{1}\oplus \mathcal{R}_{2}$ for every $x\in X$.

\end{corollary}

\begin{proof}
The only modification in the proof of Theorem \ref{thm4} is that if $A$ is norm closed, Remark \ref{rem3} (i) (instead of (ii)) can be applied.
\end{proof}

If $A$ has spectral synthesis (see \cite{Beckhoff, O'Farrell} for examples of abelian algebras that have spectral synthesis) then $\mathcal{R}_{j}=(0)$ for all $j\in\{1,2,...,p\}$ and we have the following:

\begin{corollary}
Let $A\subseteq B(X)$ be an abelian, unital, strongly closed subalgebra with spectral synthesis. If $s.s.m.(A)=p\in \mathbb{N}$, then $A$ is reflexive.
\end{corollary}

\begin{corollary}
Let $A\subseteq B(X)$ be an abelian, unital, norm closed subalgebra with spectral synthesis. If $s.s.m.(A)=2$, then $A$ is reflexive.
\end{corollary}

We will study next the hereditary reflexivity of abelian algebras of finite split strict multiplicity $A\subseteq B(X)$. We will assume that $A$ is semisimple, i.e. that the intersection of all maximal, modular ideals is $(0)$.

It is known from \cite{Rickart} that any semisimple, abelian, Banach algebra has a unique norm topology. Notice that if $A$ is semisimple, the quotients of $A$ may not be semisimple (see \cite{Dixon}).

\begin{theorem}
Let $A\subseteq B(X)$ be an abelian, unital, semisimple, strongly closed subalgebra with $s.s.m.(A)=p\in \mathbb{N}$. Let $L\subseteq A$ be a norm closed subspace. If $T\in B(X)$ is such that $Tx\in \overline{Lx}$ for every $x\in X$, then there exists $l_{0}\in L$ such that $(T-T_{l_{0}})(x)\in \underset{j=1}{\overset{p}\sum}\oplus \mathcal{R}_{j}$ for every $x\in X$.
\end{theorem} 

\begin{proof}
By Remark \ref{rem3} (ii) $D_{p}$ is closed. By Proposition \ref{ssmprop} we can assume that $X=\underset{j=1}{\overset{p}\sum}\oplus A/I_{j}$ and $D_{p}=\{\underset{j=1}{\overset{p}\sum}\oplus \varphi_{j}(a)|\; a\in A\}$.

For $a\in A$ set
$$|||a|||=\underset{j=1}{\overset{p}\sum}\|\varphi_{j}(a)\|.$$

It is easy to see that $|||\cdot |||$ is an algebra norm on $A$. Since $D_{p}$ is closed, $|||\cdot |||$ is a Banach algebra norm. By (\cite{Rickart} Corollary 2.5.18) A has a unique norm topology. Hence $|||\cdot |||$ is equivalent to the original norm of $A\subseteq B(X)$.

Let now $x_{0}=\underset{j=1}{\overset{p}\sum}\oplus \varphi_{j}(1)$. Then, in particular, $Tx_{0}\in \overline{Lx_{0}}$. There exists a sequence $\{l_{n}\}\subset L$ such that $Tx_{0}=\underset{n}{lim}\underset{j=1}{\overset{p}\sum}\oplus \varphi_{j}(l_{n}$).

Since $D_{p}$ is closed, $l_{n}\overset{|||\cdot |||}{\longrightarrow}l_{0} \in A$. Therefore $l_{n}\to l_{0} \in A$ in the original norm of $A$. Since $L$ is closed, $l_{0}\in L$. Hence
$$Tx_{0}=T(\underset{j=1}{\overset{p}\sum}\oplus \varphi_{j}(1))=\underset{j=1}{\overset{p}\sum}\oplus \varphi_{j}(l_{0}).$$

Repeating now the proof of Theorem \ref{thm4}, we get $(T-T_{l_{0}})(x)\in \underset{j=1}{\overset{p}\sum}\oplus \mathcal{R}_{j}$ for every $x\in X$.

\end{proof}

Using now Remark \ref{rem3} (i) instead of (ii), we get:

\begin{corollary}
Let $A\subseteq B(X)$ be an abelian, unital, semisimple, norm closed subalgebra with $s.s.m.(A)=2$. Let $L$ be a norm closed subspace. If $T\in B(X)$ is such that $Tx\in \overline{Lx}$ for every $x\in X$, then there exists $l_{0}\in L$ such that $(T-T_{l_{0}})(x)\in \mathcal{R}_{1}\oplus \mathcal{R}_{2}$ for every $x\in X$.
\end{corollary}

If $\mathcal{R}_{j}=(0)$ for all $j=1,2,...,p$, in particular if $A$ has spectral synthesis, we get:

\begin{corollary}
Let $A\subseteq B(X)$ be an abelian, unital, strongly closed subalgebra with spectral synthesis. If $s.s.m.(A)=p\in \mathbb{N}$, then $A$ is hereditarily reflexive.
\end{corollary}

\begin{corollary}
Let $A\subseteq B(X)$ be an abelian, unital, norm closed subalgebra with spectral synthesis. If $s.s.m.(A)=2$, then $A$ is hereditarily reflexive.
\end{corollary}

\end{document}